\documentclass[12pt]{amsart}

\usepackage{amsmath,amssymb,amsfonts,amscd,pictexwd,dcpic}

\newtheorem{theorem}{Theorem}[section]

\newtheorem{proposition}[theorem]{Proposition}

\newtheorem{corollary}[theorem]{Corollary}

\newtheorem{note}[theorem]{Note}

\begin{document}
	\title[$K$ block set partition patterns and statistics]{$K$ block set partition patterns and statistics}

	\author{Amrita Acharyya,  Robinson Paul Czajkowski, Allen Richard Williams}
	\address{Department of Mathematics and Statistics\\
		University of Toledo, Main Campus\\
		Toledo, OH 43606-3390}
	\email{amrita.acharyya@utoledo.edu}
	\email{robinson.czajkowski@rockets.utoledo.edu}
	\email{allen.williams2@rockets.utoledo.edu}
	
	
	
	\begin{abstract}
		A set partition $\sigma$ of $[n]=\{1,\cdots ,n\}$ contains another set partition $\omega$ if a standardized restriction of $\sigma$ to a subset $S\subseteq[n]$ is equivalent to $\omega$. Otherwise, $\sigma$ avoids $\omega$. Sagan and Goyt have determined the cardinality of the avoidance classes for all sets of patterns on partitions of $[3]$. Additionally, there is a bijection between the set partitions and restricted growth functions (RGFs). Wachs and White defined four fundamental statistics on those RGFs. Sagan, Dahlberg, Dorward, Gerhard, Grubb, Purcell, and Reppuhn consider the distributions of these statistics over various avoidance classes and they obtained four variate analogues of the previously cited cardinality results. They did the first thorough study of these distributions. The analogues of their many results follows for set partitions with exactly $k$ blocks for a specified positive integer $k$. These analogues are discussed in this work.
	\end{abstract}
	
	\maketitle
	
	
	\section{Introduction} \label{s:Intro}
	
	Recently, a large number of papers dealing with pattern containment and avoidance of a variety of combinatorial structures have been published. Many of these papers study various statistics on these structures. This paper is an analogue to the work done by Sagan, Dahlberg, Dorward, Gerhard, Grubb, Purcell, and Reppuhn in \cite{latexcompanion}. This paper explores pattern avoidance in set partitions with $k$ blocks on four statistics defined by Wachs and White \cite{stirling} and a variation of them as well. Sagan, Dahlberg, Dorward, Gerhard, Grubb, Purcell, and Reppuhn did the first comprehensive study of these statistics on avoidance classes in \cite{latexcompanion}. In particular, they consider the distribution of these statistics over every class avoiding a set of partitions of $\{1, 2, 3\}$ and $\{1,2,3,4\}$. Additionally, in this paper we discuss analogous results of closely related statistics. Also, later in a more recent paper \cite{RGF}], the same four statistics were discussed on RGFs that avoid words rather than partitions. 
	
	Now, we must give proper definitions similar to \cite{latexcompanion}. An attempt has been made to keep notation as close to \cite{latexcompanion} as possible.
	First, we will define a set partition. If we consider a set $S$, a partition $\pi$ is a set of disjoint subsets of $S$ whose union is $S$. Then we can write $\pi = B_1/B_2/\cdots /B_k\vdash
	S$ where each $B_i$ is called a block. There is a special way of writing set partitions where we drop the inner curly braces and remove commas. For example, if $\pi = \{1,6\}/\{2,3,4,7,8\}/\{5\}$ is a partition of $[8]$, we would write $\pi = 16/23478/5 \vdash [8]$. $\Pi_n$ is the set of all partitions of $[n]$. Formally, $\Pi_n = \{\pi | \pi \vdash [n]\}$. Now we must define containment and avoidance. Let $\pi$ be a partition in $\Pi_n$. Consider arbitrary subset $S \in [n]$. Then let $\pi'$ be $\pi$ constrained to only the elements present in $S$. In other words, $\pi'$ is the set of all non-empty intersections of $S$ and each $B_i$ in $\pi$. Then we standardize the result by replacing the smallest element in $\pi'$ with 1, the second smallest with 2, and so on. Using the above example $\pi = 16/23478/5$, if $S=\{2, 4, 6, 7\}$ then $\pi'$ is $247/6$, and the standardized $\pi'$ is $124/3$. We say $\pi$ contains the pattern $\omega$ if for some $S \in [n]$ the corresponding standardized $\pi'$ is equivalent to $\omega$. Otherwise, $\pi$ avoids $\omega$. 
	
	Now we define $\Pi_n(\omega)$ as the set of all partitions of $[n]$ that avoid $\omega$. $\Pi_n(\omega) = \{\pi \in \Pi_n | \pi $ avoids $ \omega\}$ An interesting and important fact is that there is a bijection between $\Pi_n$ and restricted growth functions (RGFs) of length $n$. In this paper we will be often converting set partitions into RGFs to analyze them. Additionally, the four statistics defined by Wachs and White in \cite{stirling} are found by using the RGF of a particular partition. A restricted growth function is a sequence of integers $w$ with the following two restrictions:\\
	
\begin{enumerate}
		\item  $a_{1} =1$, and\\

		\item  For $i\geq 2$ we have $a_{i} \leq 1+\max\{a_{1},\cdots ,a_{i-1}\}$
	\end{enumerate}
	
	For example, regarding the partition above, $\pi = 16/23478/5$, its RGF $w(\pi) = 12223122$.
	The length of a word $w$ is defined as the total number of letters it contains. Note that the maximal letter in an RGF is equal to the number of blocks $k$ of its corresponding set partition. Therefore, we can define $R_{n,k}$ as the set of all RGFs that have length $n$ and maximal letter $k$. Let $\Pi_{n,k}$ be the set of all words in $\Pi_{n}$ with exactly $k$ blocks. Let $\Pi_{n,k}(\omega)$ be the set of all words in $\Pi_{n}$ with exactly $k$ blocks that avoids $\omega$. This is useful because the cardinality of $R_{n,k}$ equals the cardinality of $\Pi_{n,k}$. It should be mentioned that when we consider set partitions, there is a specific order the blocks must be written. We always choose the first block to be the block containing $1$, and in general, $min(B_1)<min(B_2)<\cdots <min(B_k)$. For this entire paper, we assume all partitions are written in standard form.
	Sagan \cite{Pattern} first described the set partitions in $\Pi_{n}(\pi)$ for each $\pi \in \Pi_{3}$. 
	Some additional terminology would be helpful. An initial run of a word $w$ is the longest sequence of strictly increasing integers at the beginning of $w$. For example, the word $123421$ has an initial run of 4. If $a$ is a letter, we use $a^l$ to represent $l$ many consecutive $a$ in a word. A word is layered if it is of the form $1^l_12^l_2\cdots k^l_k$. Meaning it never goes down from letter to letter.

	The following properties are direct analogues to the properties presented in \cite{latexcompanion}.
	
	\begin{theorem} The characterizations below follow from theorem $1.2$ in  \cite{latexcompanion}. 
		
		\begin{enumerate}
			\item[i.] $R_{n,k}(1/2/3)= \{1\}$, if $k=1$ and this is equal to $\{w\in R_{n,k}:w$ consists of only 1 s and 2 s$\}$, if $k=2$. $R_{n,k}(1/2/3)= \emptyset$, if $k>2$.
			\item[ii.] $R_{n,k}(1/23) =\{w\in R_n : w $ is obtained  by inserting  a  single $1$ into  a  word of the form  $1^{l}23\cdots  k$ for some $l\geq 0\}$
			\item[iii.] $R_{n,k}(13/2)$  = $\{w\in R_{n}: w$ is layered ie $w = 1^{l_{1}}2^{l_{2}}\cdots  k^{l_{k}}\}$, for some positive integers $l_{1}, l_{2},\cdots  l_{k}$.
			\item[iv.]  $R_{n,k}(12/3) =\{w \in R_n : w$ has initial run $1...k$ and $a_{k+1}=\cdots  a_{n} \leq k\}$.
			\item[v.] $R_{n,k}(123) =\{w\in R_{n,k}:w$ has no element repeated more than twice$\}$
		\end{enumerate}
		
	\end{theorem}		
	
	\begin{corollary}
		
		The cardinality results that are analogues to those, as in corollary 1.3 in \cite{latexcompanion} follow.
		
		\begin{enumerate}
			\item[i.]$\# \Pi_{n,2}(1/2/3) = 2^{n-1}-1$
			\\ For any $k <m, \#\Pi_{n,k}(1/2/3)=\#\{w: w$  consists of $1,2,3 \cdots k\}$ which is 
			$\begin{Bmatrix}
				n\\
				k
			\end{Bmatrix}$, Stirling number of second kind.
			 
			Note: $\#\Pi_{n}(1/2/3)= \sum_{k=1}^{n-1}\#\Pi_{n}(1/2/3)$
			\item[ii.]$\# \Pi_{n,k}(13/2) = \binom{n-1}{k-1}$ for $k\geq 2$.
			\item[iii.]$\# \Pi_{n,k}(1/23) = \#\Pi_{n,k}(12/3)= k$, if $k<n$ and this cardinality is $1$ if $k=n$.
			\item[iv.]$\# \Pi_{n,k}(123)=1$ when $n=k$, otherwise,\\
			$\#\Pi_{n,k}(123)= \binom{n} {2n-2k} (2n-2k)!!$,\\
			where $(2n-2k)!! = (1)(3)(5)\cdots (2n-2k-1)$.
		\end{enumerate}
	\end{corollary}
	
	Next, we will define the four statistics of Wachs and White. They are denoted as $lb$, $ls$, $rb$, $rs$ which stands for \lq\lq left bigger\rq\rq,  \lq\lq left smaller\rq\rq, \lq\lq right bigger\rq\rq and \lq\lq right smaller\rq\rq respectively. Consider a word $w=a_1a_2\cdots a_n$. $lb(a_j) = \#\{a_i | i < j$ and $a_i > a_j\}$. In words, $lb(a_j)$ gives the number of distinct letters to the left of $a_j$ and bigger than $a_j$. The other three statistics are defined analogously. For example, if $w =12332412$, then $lb(a_7) =3$ since there is $a_2=2, a_3=3$ and $a_6=4$ to the left of $a_7 =1$. The statistic $lb$ of a word, $lb(w) = lb(a_1) + lb(a_2) + \cdots + lb(a_n)$. Continuing the above example, $lb(w) =0+0+0+0+1+0+3+2=6$. We sometimes write $lb(\pi)$ instead of $lb(w(\pi))$ . Following analogous notation from \cite{latexcompanion}, our main objects of study will be the generating functions
	$LB_{n,k}(\omega)= LB_{n,k}(\omega,q)=\sum_{\pi\in \Pi_{n,k}(\omega)} q^{lb(\pi)}$ and the three other analogous polynomials for the other statistics. 
	As in \cite{latexcompanion} often,   the multivariate generating function $F_{n,k}(\omega)=F_{n,k}(\omega,q,r,s,t)=  \sum_{\pi\in \Pi_{n,k}(\omega)} q^{lb(\pi)}r^{ls(\pi)}s^{rb(\pi)}t^{rs(\pi)}$ is computed.

	
	\section{The pattern 1/2/3} \label{The pattern 1/2/3}
	
	We will first describe $1/2/3$ by giving the four variable generating function $F_{n,k}(1/2/3)$ and also finding each single variable generating function of each individual statistics.
	
	\begin{theorem}
		We have
		\begin{enumerate}
			\item [i.] $F_{n,1}(1/2/3)=1$.\\
			
			\item [ii.] $F_{n,2}(1/2/3)=\sum^{n-1}_{l=1}r^{n-l}s^{l}+$\\
			$\sum^{n-1}_{l=2}\sum^{n-l-1}_{h=0}\sum_{i,j\geq 1}\binom{n-i-j-h-2}{l-i-j} q^{l-i}r^{n-l}s^{l-\delta_{h,0}j}t^{n-l-h}$,\\
			where $\delta_{h,0}$ is the Kronecker delta.
		\end{enumerate}
		
	\end{theorem}
	
	The proof follows from the proof of theorem $2.1$ in \cite{latexcompanion} 
	
	\begin{corollary}
		We have 
		\begin{enumerate}
			\item[i.] $LB_{n,1}(1/2/3) = RB_{n,1}(1/2/3) = RS_{n,1}(1/2/3) = LS_{n,1}(1/2/3) =1$
			
			\item[ii.] $LB_{n,2}(1/2/3)= RS_{n,2}(1/2/3) = \sum^{n-2}_{h=0} \binom{n-1}{h+1}q^{h}$\\
			
			\item[iii.] $LS_{n,2}(1/2/3)= (r+1)^{n-1}-1 = RB_{n,2}(1/2/3)$.
		\end{enumerate}
	\end{corollary}
	
	The proof follows by specialization of theorem $2.1$ and corollary $2.2$ in the paper \cite{latexcompanion}.
	
	
	\section{ The pattern 1/23} \label{ The pattern 1/23}
	
	In this section we find $F_{n,k}(1/23)$ and the generating functions for all four statistics. It turns out lb and rs are equal for any $w \in R_{n,k}(1/23)$.\\
	\begin{theorem}The generating functions are given by \\
		\begin{enumerate} 
			\item[i.]$F_{n,k}(1/23) = (rs)^{\binom {k}{2}}$, when $n=k$,\\
			
			\item[ii.]$F_{n,k}(1/23) = \sum^{k}_{j=1}(qt)^{j-1}r^{\binom{k}{2}},  s^{(n-k)(k-1)+(k-j)+\binom{k-1}{2}}$, when $n>k$. 
		\end{enumerate}
	\end{theorem}
	\begin{proof}
		The proof follows from theorem $3.1$ in  \cite{latexcompanion} since the maximal letter in $R_{n,k}(1/23)$ is $k$.
	\end{proof}
	
	\begin{corollary}
		We have 
		\begin{enumerate}
			\item[i.]$LB_{n,k}(1/23)=1=RS_{n,k}(1/23)$, when $n=k$, otherwise\\
			
			\item[ii.] $LB_{n,k}(1/23)=\sum^{k}_{j=1}q^{j-1}= RS_{n,k}(1/23)$.\\
			
			\item[iii.]$LS_{n,k}(1/23) = r^{\binom{k}{2}}$  when $n=k$, otherwise $LS_{n,k}(1/23) =kr^{\binom{k}{2}}$ \\
			
			\item[iv.]$RB_{n,k}(1/23) =s^{\binom{k}{2}}$, when $n=k$, otherwise $RB_{n,k}(1/23) =\sum^{k}_{j=1}s^{(n-k)(k-1)+(k-j)+\binom{k-1}{2} }$\\
			
		\end{enumerate}
	\end{corollary}
	The proof follows by  corollary $3.2$ in the paper \cite{latexcompanion} and by specialization of theorem $3.1$ above.

	
	\section{The pattern 13/2} \label{The pattern 13/2}
	
	In this section we will obtain $F_{n,k}(13/2)$ and thus the generating functions for all four statistics.
	The set of all integer partitions with exactly $k$ distinct parts of size at most $n-1$ (as in $D_{n-1}$ in \cite{latexcompanion}) is denoted by $D_{n-1,k}$. We will use this notation in our result.
	\begin{theorem} We have $F_{n,k}(13/2)= \sum_{\lambda \in D_{n-1,k}}r^{|n-\lambda|}s^{|\lambda|}$.
	\end{theorem}		
	The proof follows from the proof of theorem $4.1$ in the paper \cite{latexcompanion}.\\
	
	The generating function of each individual statistic is easy to obtain by specialization of Theorem $4.1$.

\newpage
	
	\begin{corollary} we have
		\begin{enumerate}
			\item[i.] $LB_{n,k}(13/2)= \binom{n-1}{k-1}= RS_{n,k}(13/2)$.\\
			
			\item[ii.] $LS_{n,k}(13/2)= \sum_{\lambda \in D_{n-1,k}}r^{|n-\lambda|}$\\
			
			\item[iii.] $RB_{n,k}(13/2)= \sum_{\lambda \in D_{n-1,k}}s^{|\lambda|}$
		\end{enumerate}
	\end{corollary}
	\section{The pattern 12/3}\label{The pattern 13/2}
	
	In this section we determine $F_{n,k}(12/3)$. The other polynomials associated with 12/3 are obtained as corollaries.\\ 
	\begin{theorem}
		We have
		\begin{enumerate}
			\item[i.]$F_{n,k}(12/3) = (rs)^{\binom{k}{2}}$, when $n=k$, otherwise,\\
			
			\item[ii.]  $ F_{n,k}(12/3) =\sum^{k}_{i=1}q^{(n-k)(k-i)}r^{\binom{k}{2}+(n-k)(i-1)}s^{\binom{k}{2}}t^{k-i}$.\\
			
		\end{enumerate}
	\end{theorem}
	Proof follows from the proof of theorem $5.1$ in the paper \cite{latexcompanion}.
	\begin{corollary}
		
		We have
		\begin{enumerate}
			\item[i.] $LS_{n,k}(12/3) = r^{\binom{k}{2}}$, when $n=k$, otherwise\\
			
			\item[ii.] $LS_{n,k}(12/3)=\sum^{k}_{i=1}r^{\binom{k}{2} +(n-k)(i-1)}$\\
			
			\item[iii.] $RB_{n,k}(12/3) = s^{\binom{k}{2}}$, when $n=k$, otherwise $RB_{n,k}(12/3) =ks^{\binom{k}{2}}$\\
			
			\item[iv.] $RS_{n,k}(12/3)=1$ if $n=k$, otherwise, $RS_{n,k}(12/3)= \sum^{k-1}_{i=1}t^{k-i}$ 
		\end{enumerate}
	\end{corollary}
	
	Proof follows from corollary $5.2$ in the paper \cite{latexcompanion} and by specialization of theorem $5.1$.
	
	\begin{proposition} We have\\
		$LB_{n,k}(12/3)= \sum^{(n-k)(k-1)}_{i=0}D_{i}q^{i}$ where  $$D_{i}= \#\{d\geq 1: d|i, d+\frac{i}{d}+1\leq n\}$$ 
	\end{proposition}
	
	The proof follows by the proof of proposition $5.3$ in the paper \cite{latexcompanion}.\\
	
	The notation $D_i$ is used by \cite{latexcompanion} to find $LB_n(12/3)$. We write the analogue here, but we can find a much more elegant solution if we take the number of blocks $k$ into consideration. \\
	
	$LB_{n,k}(12/3) = \sum^k_{i=1}q^{(n-k)(k-i)}$\\
		
		The final result of this section provide two interesting relationships between the avoidance classes $\Pi_{n,k}(1/23)$ and 
		$\Pi_{n,k}(12/3)$.
		
		\begin{proposition}
			We have
			\begin{enumerate}
				\item[i.] $LB_{n,k}(1/23) = RS_{n,k}(12/3)$\\
				
				\item[ii.]$LS_{n,k}(1/23) = RB_{n,k}(12/3)$
			\end{enumerate}
		\end{proposition}
		\begin{proof}
			The proof follows from the proof of Proposition $5.5$ in the paper \cite{latexcompanion} 
		\end{proof}

		
		\section{The pattern 123} \label{The pattern 123}
		
		As in \cite{latexcompanion}, for the previous four partitions of \cite{Q-count}, we find a $4$-variable generating function describing all four statistics on the avoidance class. The pattern $123$, however, is much more difficult, so we do not find a generating function including all four variables. We instead give results for the generating functions of a single variable as in \cite{latexcompanion}. Consider the left-smaller statistic first.
		\begin{theorem}
			We have \\
			$LS_{n,k}(123) = \sum _{L}\bigl(\Pi^{n-k}_{g=1}(k-l_{g}+g)\bigr)q^{\binom{k}{2} +\sum_{l\in L}{\begin{array}{cc}(l-1)\end{array}}}$, where the sum is over all subsets $L = \{l_{1}, l_{2},\cdots \cdots  l_{n-k}\} $ of $[k]$ with $l_{1}>l_{2}>\cdots >l_{n-k}$.
		\end{theorem}
		
		\begin{proof}
			It follows from Theorem $1.2$ that $R_{n,k}(123)$ is nonempty, whenever $k\geq \lceil{\frac{n}{2}}\rceil$. The proof follows from the proof of theorem $6.1$ in the paper \cite{latexcompanion}.
		\end{proof}		
		\begin{theorem} For $k\geq \lceil{\frac{n}{2}}\rceil$
			\begin{enumerate}
				\item[i.] The degree of $LB_{n,k}(123)= \frac{(4n+1)k-3k^{2}-n^{2}-n}{2}$.\\
				
				\item[ii.] The leading coefficient of $LB_{n,k}(123)$ is $(n-k)!$.
			\end{enumerate}
		\end{theorem}
		
		\begin{theorem}
			
			\end{theorem}

		The constant term of $LB_{n,k }(1234\cdots m)$:
		Since for any $k\leq m,$ $R_{n,k}(1234\cdots m)=\{w:$ has no element repeated more than $m-1$ times\}, it is the number of compositions of $n$ in to exactly $k$ distinct parts, where each part has size at most $m-1$ and that is given by the coefficient of $x^n$ in $\left(x \frac{1-x^m}{1-x}\right)^k$, ~\cite{Comb}.
		
\newpage

		\begin{theorem}
			
			We have
			\begin{enumerate}
				\item[i.] The degree of $RS_{n,k}(123)$ is $(n-k)(k-1)$.
				
				\item[ii.] Since given $k$ is fixed, it follows that the leading coefficient of $RS_{n,k}(123)$ is $1$.
				
				\item[iii.] The constant term of $RS_{n,k}(123)$ is same as that of $RS_{n,k}(123)$.
				
				\item[iv.] $RB_{n,k}$ is monic. In $R_{n,k}(123)$, a layered term of the form $1^{2} 2^{2}\cdots$\\
				$(n-k)^{2}(n-k+1)\cdots k$ maximizes $rb$. So degree of $RB_{n,k}(123)$ is $\binom{k}{2}$. 
			\end{enumerate}
		\end{theorem}
		\begin{proof}
		It follows from the proof of theorem 6.4 and 6.5 in \cite{latexcompanion}.
		\end{proof}

		\begin{table}[ht]
			\centering 
			\begin{tabular}{||c|| c||} 
				\hline
				Avoidance class& Associated RGF's \\ [0.5ex] 
				\hline\hline
				$\Pi_{n,k} (1/2/3, 1/23)$ &  $1^{n} (k=1) , 1^{n-1}2, 1^{n-2}21 (k=2)$ \\   
				\hline
				$\Pi_{n,k} (1/2/3, 13/2)$ & $1^{m}2^{n-m}, 1\leq m\leq n (k=2)$ \\  
				\hline
				$\Pi_{n,k} (1/2/3, 12/3)$ & $1^{n} (k=1), 12^{n-1}, 121^{n-2} (k=2)$\\ 
				\hline
				$\Pi_{n,k} (1/23, 13/2)$ & $1^{n-k+1}23\cdots k$\\ 
				\hline
				$\Pi_{n,k} (1/23, 12/3)$ & $1^{n}(k=1), 123\cdots (n-1)1(k = n-1)$\\
				$\phantom e$&  $123\cdots n (k=n)$\\ 
				\hline 
				$\Pi_{n,k} (1/23,123)$ & $123\cdots (n-1)$ with an additional\\
				$\phantom e$& $1$ inserted $(k=n-1)$, $123\cdots  n (k=n)$\\ 
				\hline $\Pi_{n,k}  (13/2, 12/3)$& $123\cdots k^{n-k+1}$\\
				\hline $\Pi_{n,k}(13/2, 123)$ & Layered RGF's,\\
				$\phantom e$& each layer with  at most $2$ elements, $ k \geq \lceil{\frac{n}{2}}\rceil $ \\
				\hline $\Pi_{n,k}(12/3,123)$ & $123\cdots k (k=n), 123\cdots ki,  1\leq i\leq k (k=n-1)$\\
				[1ex]
				\hline
				\end{tabular}
				\caption{\scriptsize {Avoidance classes avoiding two partitions of $[3]$ and associated RGFs of exactly $k$ blocks.}}
				\end{table}

				
				\section{Multiple pattern avoidance:}
				
				In this section, multiple pattern avoidance is explored. If $P$ is a set of partitions in $\Pi_3$, we define $\Pi_{n,k}(P)$ as the set of all partitions in $\Pi_{n,k}$ that avoids every partition in $P$. Goyt \cite{set} characterized that cardinalities of $\Pi_{n,k}(P)$ for any $P \subseteq S_{3}.$ We will do the same for $F_{n,k}(P)$ as $F_{n}(P)$ was done in \cite{latexcompanion}. Table 1 is the characterization of the RGFs of each $\Pi_{n,k}(P)$ for $P$ size 2. This table is a result from Goyt \cite{set} and was also presented in \cite{latexcompanion} for completeness. In this section, we assume $n\geq 3$ because for $n<2$, $\Pi_{n,k}(P)=\Pi_{n,k}$. After that, we give the generating functions $F_{n,k}$ of these RGFs.

				From theorem $7.1$ and the fact that the number of blocks $k$ is same with the maximal letter in a word $w \in R_{n,k}(P)$ where P is a set of set partitions of $[n]$, we have the following analogues of Theorem $7.1$ in the paper\cite{latexcompanion}.
				
				\begin{theorem}
				For $n\geq 3$
				\begin{enumerate}
				\item[1.] $F_{n,k}(1/2/3, 1/23) = 1$,when $k = 1$ and it is equal to $rs^{n-1}+qrs^{n-2}t$ when $k=2$.\\
				
				\item[2.] $F_{n,k}(1/2/3,13/2) = 1$ when $k = 1$, and it is equal to $\sum^{n-1} _{i=1}r^{i}s^{n-i}$, otherwise.\\
				
				\item[3.] $F_{n,k}(1/2/3,12/3) = 1$ when $k = 1$, and it is equal to $r^{n-1}s+q^{n-2}rst$, when $k =2$.\\
				
				\item[4.]  $F_{n,k}(1/23,13/2) = 1$ when $k=1$, and it equals to $r^{\frac{k(k-1)}{2}}s^{{\frac{k-1}{2}}(2n-k)}$, otherwise.\\
				
				\item[5.] $F_{n,k}(1/23,12/3)=1$ for $k =1$, and it is equal to $(rs) ^ {\binom{k}{2}}$ for $k= n$, and is equal to $  (qt)^{k-1}(rs) ^ {\binom{k}{2}}$ for $k=n-1$.\\
				
				\item[6.] $F_{n,k}(1/23,123)= (rs) ^ {\binom{k}{2}}$ when  $n = k$ and this is equal to $ r^{\binom{n-1}{2}} \sum^{n-2}_{i=0}(qt)^{i}s^{\binom{n}{2}-i-1}$, for $k=n-1$.\\
				
				\item [7.]  $F_{n,1}(13/2,12/3)= 1$, when $k=1$, otherwise, it is equal to $r^{\frac{(k-1)(2n-k)}{2}}s^{\frac{(k-1)k}{2}}$.\\
				
				\item [8.] $F_{n,k}(13/2,123)= r^{\binom{k}{2}+\sum_{l\in L}{\begin{array}{cc}(l-1)\end{array}}}s^{\binom{k}{2}+\sum_{l\in L}{\begin{array}{cc}(k-l)\end{array}}}$, where $L$ is over all subsets $L = \{l_{1}, l_{2},\cdots  l_{n-k}\} $ of $[k]$ with $l_{1}>l_{2}>\cdots >l_{n-k}$.\\
				
				\item [9.] $F_{n,k}(12/3,123)=(rs)^ {n(n-1)/2}$ when $n=k$, $F_{n,k}(12/3,123)=s^{\binom{n-1}{2}} \sum^{n-2}_{i=0}(qt)^{i}r^{\frac{{n}{(n-1)}}{2}-i-1}$, for $k=n-1$.
				\end{enumerate}			
				\end{theorem}

				\begin{table}[ht]
				\centering 
				\begin{tabular}{||c ||c||} 
				\hline
				Avoidance class& Associated RGF's  \\ [0.5ex] 
				\hline\hline
				$\Pi_{n,k} (1/2/3, 1/23, 13/2)$ &  $1^{n} (k=1) , 1^{n-1}2 (k=2)$  \\ 
				\hline
				$\Pi_{n,k} (1/2/3,1/23,12/3)$ & $1^{n} (k=1), 121 (k=2, n=3)$   \\
				\hline
				$\Pi_{n,k} (1/2/3, 12/3, 13/2)$ & $1^{n} (k=1), 12^{n-1} (k=2)$ \\
				\hline
				$\Pi_{n,k} (1/23,12/3, 13/2)$ & $1^{n} (k=1), 123...k (k=n)$. \\
				\hline
				$\Pi_{n,k} (1/23, 13/2,123)$ & $123\cdots  k, (k=n)$,\\
				$\phantom e$& $1^{2}23\cdots (n-1), (k = n-1)$\\ 
				\hline $\Pi_{n,k} (1/23,12/3,123)$  & $123\cdots  k, (k=n)$,\\
				$\phantom e$& $123\cdots (n-1)1, (k=n-1)$\\
				\hline $\Pi_{n,k}(13/2, 12/3, 123)$ &  $123\cdots  k, (k=n)$,\\
				$\phantom e$& $123\cdots (n-2)(n-1)^{2}, (k=n-1)$\\
				\hline $\Pi_{n,k}(1/2/3, 1/23,12/3,13/2)$ &  $1^{n} (k=1)$\\
				\hline $\Pi_{n,k}(123, 13/2, 1/23,12/3)$ & $123\cdots  k (k=n)$\\
				[1ex]
				\hline
				\end{tabular}
				\caption{\scriptsize {Avoidance classes and associated RGF's avoiding three and four partitions of $[3]$}}
				\end{table}

				As a direct analogue of Corollary $7.2$ in \cite{latexcompanion} we have the following:
				
				\begin{corollary}
				Consider the generating function $F_{n,k}(P)$, where $P \subseteq \Pi_{3}$.
				\begin{enumerate} 
				\item[1.] $F_{n,k}(P)$ is invariant under switching q and t if $13/2\in P$ or P is one of $\{1/2/3,1/23\}; \{1/23,12/3\}; \{1/23,123\}; \{12/3,123\}$.
				\item[2.] $F_{n,k}(P)$ is  invariant under switching r and s if P is one of $\{1/2/3,13/2\}; \{1/23,12/3\}$. 
				\item[3.] The following equality between generating functions for different P follows: $F_{n,k}(1/23,13/2;q,r,s,t)= F_{n,k}(13/2,12/3;q,s,r,t)$ and $F_{n,k}(1/23,123;q,r,s,t) =F_{n,k}(12/3,123;q,s,r,t)$.
				\end{enumerate}
				\end{corollary}
				
				Next, we will examine the outcome of avoiding three and four partitions of $[3]$. We can see the avoidance classes and the resulting restricted growth functions in Table 2. The entries in this table can easily be turned into a polynomial by the reader if desired. Avoiding all five partitions of $[3]$ is not identified, because it would obtain both $1/2/3$ and $123$.\\
				
				\section{A variation on the statistics of Wachs and White}
				
				It is interesting to consider the effect on the generating functions of allowing equality in the four statistics given by Wachs and White. We consider $lbe, lse, rbe$ and $rse$ where \lq\lq l\rq\rq stands for \lq\lq left\rq\rq, \lq\lq r\rq\rq stands for \lq\lq right\rq\rq, \lq\lq b\rq\rq stands for \lq\lq bigger\rq\rq, and \lq\lq s\rq\rq stands for \lq\lq smaller\rq\rq and \lq\lq e\rq\rq stands for \lq\lq equal\rq\rq. The left-bigger or equal statistic is described here. Given a word $w =a_{1}\cdots  a_{n}$  define $lbe(a_{j}) =\#\{a_{i} : i \leq j$ and $a_{i} \geq a_{j}\}$.
				In words, the set of integers occurring before $a_{j}$ and bigger than or equal to $a_{j}$ are counted. It is important to note that, the cardinality of a set is taken, so if there are multiple copies of such an integer then it is only counted once. For example, if $w =12332412$, then $lbe(a_7) =4$ since $a_1=1, a_2=2, a_3=3$ and $a_6=4$ to the left of $a_7 =1$. Finally, define $lbe(w) =lbe(a_1)+\cdots +lbe(a_n)$. Continuing the above example, $lbe(12332412) =0+0+0+1+2+0+4+3=10$. To simplify notation, $lbe(\sigma)$ is written instead of more cumbersome $lbe(w(\sigma))$. Following analogous notation from \cite{latexcompanion} our main objects of study will be the generating functions
				$LBE_{n,k}(\pi)= LBE_{n,k}(\pi,q)=\sum_{\sigma\in \Pi_{n,k}(\pi)} q^{lb(\sigma)}$ and the three other analogous polynomials for the other statistics. 
				As in \cite{latexcompanion} often,   the multivariate generating function $FE_{n,k}(\varpi)=FE_{n,k}(\varpi,q,r,s,t)=  \sum_{\varsigma\in \Pi_{n,k}(\varpi)} q^{lb(\varsigma)}r^{ls(\varsigma)}s^{rb(\varsigma)}t^{rs(\varsigma)}$ can be computed.

				\begin{note}
				Note:
				\begin{enumerate} 
				\item[i]The generating function $FE_{n}$ can be found from $FE$ by adding $n-k$ in the exponent of each of the four variables and accordingly for the other corresponding one variable generating functions as well, where $k$ is the largest letter in the corresponding RGF.

				\item[ii]The generating function of $FE_{n,k}$ can also be found by adding $n-k$ in the exponent of each variable. And accordingly for the other one variable generating functions as well. For example it follows from the proof of theorem 6.2 in \cite{latexcompanion} that the degree of $LBE_{n,k}(123)= \frac{(4n-1)k-3k^{2}-n^{2}+n}{2}$.
				\end{enumerate}
				\end{note}
				
				\section{ The pattern 1/23} \label{ The pattern 1/23}
				
				In this section we find $FE_{n}(1/23)$ and the generating functions for all four statistics. Similarly to $lb$ and $rs$, lbe and rse are equal for any $w \in R_{n}(1/23)$.\\
				\begin{theorem}
				$FE_{n}(1/23) =$\\
$(qrst)^{(n-1)}+ (rs)^{\binom {n}{2}}+ \sum^{n-1}_{m=1}\sum^{m}_{j=1}(qt)^{j-1+n-m}r^{\binom{m}{2}+(n-m)} s^{(n-m)m+(m-j)+\binom{m-1}{2}}$.
				\end{theorem}
				\begin{proof}
				The proof follows from the proof of theorem $3.1$ in  \cite{latexcompanion} 
				\end{proof}
				
				\begin{corollary}
				We have 
				\begin{enumerate}
				\item[i.]$LBE_{n}(1/23)=RSE_{n}(1/23)= 1 + q^{n-1} + \sum^{n-1}_{m=1}(n-m)q^{n-m}$.\\
				
				\item[ii.]$LSE_{n}(1/23) = r^{n-1}+r^{\binom{n}{2}}+ \sum^{n-1}_{m=1}mr^{\binom{m}{2}+n-m}$ \\
				
				\item[iii.]$RBE_{n}(1/23) =s^{n-1}+s^{\binom{n}{2}} =\sum^{(n-1)}_{m=1}\sum^{m}_{j=1}s^{(n-m)m+(m-j)+\binom{m-1}{2} }$\\
				
				\end{enumerate}
				\end{corollary}
				The proof follows by  corollary $3.2$ in the paper \cite{latexcompanion} and by specialization of theorem $3.1$ above.
				
				\section{The pattern 12/3}
				
				In this section we determine $FE_{n}(12/3)$. The other polynomials associated with 12/3 are obtained as corollaries.\\ 
				\begin{theorem}
				We have
				$FE_{n}(12/3) =$\\
$(rs)^{\binom{n}{2}}+\sum^{n-1}_{m=1}\sum^{m}_{i=1}q^{(n-m)(m-i+1)}r^{\binom{m}{2}+(n-m)i-1}s^{\binom{m}{2}+n-m}t^{n-i}$.\\
				
				\end{theorem}
				Proof follows from the proof of theorem $5.1$ in the paper \cite{latexcompanion}.
				\begin{corollary}
				
				We have
				\begin{enumerate}
				\item[i.] $LSE_{n}(12/3) = r^{\binom{n}{2}}+\sum^{n-1}_{m=1}\sum^{m}_{i=1}r^{\binom{m}{2} +(n-m)i}$\\
				
				\item[ii.] $RBE_{n}(12/3) = s^{\binom{n}{2}}+\sum^{n-1}_{m=1}ms^{\binom{m}{2}+n-m}$\\
				
				\item[iii.] $RSE{n}(12/3)=1+ \sum^{n-1}_{i=1}t^{n-i}$ 
				\end{enumerate}
				\end{corollary}
				
				Proof follows from corollary $5.2$ in the paper \cite{latexcompanion} and by specialization of theorem $5.1$.
				
				\begin{proposition} We have\\
				$LBE_{n}(12/3)= \sum^{\lfloor  {\frac{n^{2}}{4}}\rfloor}_{k=0}D_{k}q^{k}$ where  $$D_{i}= \#\{d\geq 1: d|i, d+\frac{i}{d}+1\leq n\}$$ 
				\end{proposition}
				
				The proof follows by the proof of proposition $5.3$ in the paper \cite{latexcompanion}.\\
				
				The final result of this section provide two interesting relationships between the avoidance classes $\Pi_{n}(1/23)$ and 
				$\Pi_{n}(12/3)$ as earlier.
				
				\begin{proposition}
				We have
				\begin{enumerate}
				\item[i.] $LBE_{n}(1/23) = RSE_{n}(12/3)$\\
				
				\item[ii.]$LSE_{n}(1/23) = RBE_{n}(12/3)$
				\end{enumerate}
				\end{proposition}
				\begin{proof}
				The proof follows from the proof of Proposition $5.5$ in the paper \cite{latexcompanion} as in case of lbe, lse, rbe, rse in both type of patterns avoiding $1/23, 12/3, n-m$ is added. 
				\end{proof}

				
				\section{The pattern 123} \label{The pattern 123}
				
				As in \cite{latexcompanion}, the other four set partitions of \cite{Q-count} are much easier than $123$. Here, we are not able to find $FE_n$, so we only find the individual generating functions. Consider the left-smaller statistic first.
				\begin{theorem}
				$LSE_{n}(123) =$\\
$\sum^{n}_{m=\lceil{\frac{n}{2}}\rceil} \sum _{L}\bigl(\Pi^{n-m}_{g=1}(m-l_{g}+g)\bigr)q^{\binom{m}{2} +\sum_{l\in L}{\begin{array}{cc}l\end{array}}}$, where the sum is over all subsets $L = \{l_{1}, l_{2},\cdots \cdots  l_{n-m}\} $ of $[m]$ with $l_{1}>l_{2}>\cdots >l_{n-m}$.
				\end{theorem}
				
				\begin{proof}
				The proof follows from the proof of theorem $6.1$ in the paper \cite{latexcompanion}.
				\end{proof}		
				\begin{theorem} Degree and leading coefficient of $LBE_n(123)$
				\begin{enumerate}
				\item[i.] The degree of $LBE_{n}(123)$is $ \lfloor\frac{n(n+1)}{6}\rfloor$.\\
				
				\item[ii.] The leading coefficient of $LBE_n(123)$ is $k!$, if $n=3k$, it is $(k-1)!$ if $n=3k+2$, it is $(k+1)!$ if $n=3k+1$.
				\end{enumerate}
				\end{theorem}		
				
				The proof follows from the proof of theorem $6.2$ in the paper \cite{latexcompanion}.
				
				\begin{theorem}
				
				We have
				\begin{enumerate}
				\item[i] The constant term of $LBE_{n}(123)$ is $1$ as the constant term appears only from $123\cdots n$.
				\item[ii.] The coefficient of $q$ in
				$LBE_{n}(123)$ is $\binom{n}{2}$
				\item[iii.] The degree of $RSE_{n}(123)$ is $\lceil\frac{n^{2}}{4}\rceil$.
				\item[iv] The constant term of $RSE_{n}(123)$ is $1$ as the constant term appears only from $123\cdots n$.
				
				\item[v.] It follows that the leading coefficient of $RBE_{n}(123)$ is $2$.
				
				\end{enumerate}
				\end{theorem}
				\begin{proof}
				The proofs follows from theorem $6.4$ and theorem 6.5 in \cite{latexcompanion}) 
				\end{proof}
				
				\section{multi pattern avoidance}

				As a direct analogue of Theorem $7.1$ in \cite{latexcompanion} we have the following:
				
				\begin{theorem}
				For $n\geq 3$
				\begin{enumerate}
				\item[1.] $FE_{n}(1/2/3, 1/23) = (qrst)^{n-1}+(qt)^{n-2}r^{n-1}s^{2n-3}+(qrt)^{n-1}s^{2n-4}$\\
				
				\item[2.] $FE_{n}(1/2/3,13/2) =  (qrst)^{n-1}+\sum^{n-1} _{m=1}q^{n-2}r^{2n-m-2}s^{m+n-2}t^{n-2}$\\
				
				\item[3.] $FE_{n}(1/2/3,12/3) = (qrst)^{n-1}+q^{n-2}r^{2n-3}s^{n-1}t^{n-2}+q^{2n-4}(rst)^{n-1}$\\
				
				\item[4.]  $FE_{n}(1/23,13/2)=$\\
$(qrst)^{n-1} + \sum^{n-1}_{i=1}q^{n-i-1}r^{\binom{n-i+1}{2}+n-i-1}s^{\binom{n}{2}-\binom{i}{2}+n-i-1}t^{n-i-1}$\\

				\item[5.] $FE_{n}(1/23,123)= (rs) ^{\binom{n}{2}}+r^{\binom{n-1}{2}+1} \sum^{n-2}_{i=0}(qt)^{i+1}s^{\binom{n}{2}-i}$\\
				
				\item [6.]  $FE_{n}(13/2,12/3)=$\\
$(qrst)^{n-1}+\sum_{i=1}^{n-1}q^{n-i-1}r^{\binom{n}{2}-\binom{i}{2}+n-i-1}s^{\binom{n-i+1}{2}+n-i-1}t^{n-i-1}$\\

				\end{enumerate}			
				\end{theorem}

				As a direct analogue of Corollary $7.2$ in \cite{latexcompanion} we have the following:
				
				\begin{corollary}
				Consider the generating function $FE_n(P)$, where $P \subseteq \Pi_{3}$.
				\begin{enumerate} 
				\item[1.] $FE_{n}(P)$ is invariant under switching q and t if $13/2\in P$ or P is one of $\{1/2/3,1/23\}; \{1/23,12/3\}; \{1/23,123\}; \{12/3,123\}$.
				\item[2.] $FE_{n}(P)$ is  invariant under switching r and s if P is one of $\{1/2/3,13/2\}; \{1/23,12/3\}$. 
				\item[3.] The following equality between generating functions for different P follows: $FE_{n}(1/23,13/2;q,r,s,t)= FE_{n}(13/2,12/3;q,s,r,t)$.
				\end{enumerate}
				\end{corollary}

\section{Variation on the statistics of Wachs and White on $R_{n}(v)$ in \cite{RGF}, where $v$ is a standardized pattern} 
				
				Note that the Variation of the Statistics are defined on RGF's. So, in that set up we can consider the analogue of some results in \cite {RGF}. For example as in Proposition 2.3 in \cite{RGF} $LBE_{n}(112)= LBE_{n}(122)$. As in Proposition 2.6, in {RGF} $LSE_{n}(112)= LSE_{n}(121)$. As in Proposition 2.7 in \cite{RGF}, $LSE_{n}(112)= LSE_{n}(122)$. As in Theorem 2.9, $RBE_{n}(112)=LSE_{n}(122)= \sum^{n}_{m=0}{\binom{n-1}{n-m}}q^{\binom{m}{2}+(n-m)}$. Analogues of multiple pattern avoidance in Theorem 3.3, 4, 5, 7 in \cite{RGF} also follows accordingly.

				\bibliographystyle{amsplain}

\begin{thebibliography}{1}
				
				\bibitem{latexcompanion} 
				S. Dahlberg, R. Dorward, J. Gerhard, T. Grubb, C. Purcell, L. Reppuhn, B. E. Sagan,
				\emph{Set partition patterns and statistics}, Discrete Mathematics, 2016.
				
				\bibitem{RGF}
				S. Dahlberg, R. Dorward, J. Gerhard, T. Grubb, C. Purcell, L. Reppuhn, B. E.Sagan, \emph{Restricted growth function patterns and statistics}, Advances in Applied MathematicsPre-print, 2018.
				
				\bibitem{Q-count}
				A. M. Garsia, J. B. Remmel, \emph{Q-counting rook configurations and a formula of Frobenius}, J.Combin.Theory Ser.A41(2)(1986) 246-275.
				
				\bibitem{Avoidence}
				Adam M. Goyt, \emph{Avoidance of partitions of a three-element set}, Adv.Appl.Math.41(1)(2008)95–114.
				
				\bibitem{set}
				Adam M.Goyt. Bruce E Sagan, \emph{Set partition statistics and q-Fibonacci numbers}, European J.Combin.30 (1)(2009) 230-245.
				
				\bibitem{Rook}
				J. Haglund, J. B. Remmel, \emph{Rook theory for perfect matchings}, Adv. Appl. Math. 27 (2–3) (2001) 438–481. Special issue in honor of Dominique Foata’s 65th birthday (Philadelphia, PA, 2000)
				
				\bibitem{Surles}
				G. Kreweras, \emph{Surles partitions non croisées d’uncycle}, Discrete Math.1(4)(1972) 333-350. 
				
				\bibitem{Pattern}
				B. E. Sagan, \emph{Pattern avoidance in set partitions}, Ars Combin.94(2010) 79-96.
				
				\bibitem{Mah}
				B E. Sagan, C. D. Savage, \emph{Mahonian pairs}, J.Combin.TheorySer.A119(3)(2012) 526-545.
				
				\bibitem{Comb} 
				R. P. Stanley \emph{Enumerative Combinatorics}, Vol. 1, in: Cambridge Studies in Advanced Mathematics, vol. 49, Cambridge University Press, Cambridge,1997. With a foreword by Gian-CarloRota, Cor\underline{rec}corrected reprint of the 1986 original.
				
				\bibitem{stirling} 
				M. Wachs, D. White \emph{p,q-Stirling numbers and set partition statistics}, J.Combin.Theory Ser.A56(1)(1991) 27-46.\\
				
				
				\end{thebibliography}

				\end{document}